\documentclass[11pt,a4paper]{amsart}
\usepackage{mathrsfs}
\usepackage{syntonly}
\usepackage{amsmath}
\usepackage{amsthm}
\usepackage{amsfonts}
\usepackage{amssymb}
\usepackage{latexsym}
\usepackage{amssymb,amsopn,amsmath,amsthm,graphics,amsfonts,mathrsfs,accents,enumerate,verbatim,calc}
\usepackage{graphicx}
\usepackage[colorlinks=true,linkcolor=blue,citecolor=blue]{hyperref}
\usepackage{mathtools}
\usepackage{tikz}
\usepackage[all]{xy}
\usepackage{cleveref}
\usepackage{tkz-euclide}
\usepackage{tikz-cd}
\usetikzlibrary{cd, arrows, arrows.meta, decorations.pathmorphing}

\date{}
\allowdisplaybreaks[4] \footskip=15pt
\renewcommand{\uppercasenonmath}[1]{}

\numberwithin{equation}{section}
\theoremstyle{plain}

\newtheorem{thm}{Theorem}[section]
\newtheorem{lem}[thm]{Lemma}
\newtheorem{cor}[thm]{Corollary}
\newtheorem{prop}[thm]{Proposition}

\newtheorem{definition}[thm]{Definition}

\newtheorem{example}[thm]{Example}
\newtheorem{question}[thm]{Question}
\newtheorem{property}[thm]{Property}
\newtheorem{properties}[thm]{Properties}
\newtheorem{observation}[thm]{Observation}

\newtheorem{remark}[thm]{Remark}

\newtheorem{Subprops}{}       

\newcommand*{\Hom}{\operatorname{Hom}}

\newcommand{\Ho}{\operatorname{Ho}}

\newenvironment{rmk}{\begin{remark}\rm}{\end{remark}}

\newtheorem*{ack*}{ACKNOWLEDGEMENTS}

\newcommand{\pf}{\noindent\begin {proof}}
\newcommand{\epf}{\end{proof}}

\def\Ext{\operatorname{Ext}}

\begin{document}
\begin{center}
	{\large  \bf  Mixing Extriangulated Model Structures}
	
	\vspace{0.5cm}  Junpeng Ren,  Xianhui Fu\footnote{Corresponding author.}\\
	\medskip

\end{center}

\bigskip
\centerline { \bf  Abstract}
\medskip

\leftskip10truemm \rightskip10truemm \noindent
\hspace{1em} Let $(\mathscr{C}, \mathbb{E}, \mathfrak{s})$ be a weakly idempotent complete extriangulated category, and let $\xi_1\subseteq\xi_2$ be proper classes of $\mathbb{E}$-triangles. We prove an analogue of Cole's mixing theorem in this setting: two compatible admissible model structures, relative to $\xi_1$ and $\xi_2$ respectively, give rise to a mixed admissible model structure $\mathcal{M}_m$ relative to $\xi_1$, which has the trivial objects of the second model structure and the fibrant (or cofibrant) objects of the first. We also describe the cofibrant objects of $\mathcal{M}_m$ explicitly in terms of the two original model structures. For exact categories this recovers the mixing theorem for exact model structures, and for triangulated categories it applies to proper classes of triangles; we give a non-degenerate example on the homotopy category of a quasi-Frobenius ring.\par
\medskip
{\rightskip=10truemm plus 1fil
\noindent{\bf Keywords:} Extriangulated categories; proper classes; cotorsion pairs; Hovey triples; mixed model structures.\par
\noindent{\bf 2020 Mathematics Subject Classification:} 18N40; 18G80; 18G25; 18E10; 18G35.\par}

\leftskip0truemm \rightskip0truemm
\section{Introduction}

Since their introduction by Quillen \cite{Quillen_1967}, model categories
have provided a fundamental framework for abstract homotopy theory. In
homological algebra, Hovey \cite{hovey2002cotorsion} showed that abelian
model structures correspond to pairs of compatible complete cotorsion
pairs, now called Hovey triples; see also Christensen and Hovey
\cite{Christensen_Hovey_2002} for model structures adapted to relative
homological algebra. Gillespie \cite{gillespie2011model} extended this
correspondence to weakly idempotent complete exact categories, and Nakaoka
and Palu \cite{NP} extended it further to extriangulated categories, where
it describes the \emph{admissible} model structures; see also
\cite{Ma_You_Zhang_Zhou_2026} for admissible weak factorization systems.
Since extriangulated categories simultaneously generalize exact
categories, triangulated categories and extension-closed subcategories of
triangulated categories \cite{NP}, this gives a uniform method for
constructing model structures from cotorsion pairs in all of these
settings.

A fundamental construction in model category theory is Cole's mixing
theorem \cite{cole2006mixing}. Given two model structures $\mathcal{M}_1$
and $\mathcal{M}_2$ on the same category such that every weak equivalence
and every fibration of $\mathcal{M}_1$ is one of $\mathcal{M}_2$, it
produces a model structure with the weak equivalences of $\mathcal{M}_2$
and the fibrations of $\mathcal{M}_1$. In the basic algebraic example the
two model structures are built from different classes of extensions: on
chain complexes over a ring $R$, the Hurewicz model structure is an exact
model structure for the degreewise split short exact sequences, whereas
the projective model structure is one for all short exact sequences.
Mixing them yields a model structure whose cofibrant objects are the
K-projective complexes (Example~\ref{ex:chain-complexes}), in analogy with
the mixing of the Str{\o}m and Quillen model structures on topological
spaces, whose cofibrant objects are the spaces of the homotopy type of
CW-complexes. Gillespie \cite{Gillespie_2025} formulated the mixing
construction for exact model structures in terms of Hovey triples.

In an extriangulated category, the appropriate notion of a class of
extensions is that of a proper class of $\mathbb{E}$-triangles, introduced
by Hu, Zhang and Zhou \cite{hu2020proper} following Beligiannis' work on
triangulated categories \cite{Beligiannis2000}. A proper class $\xi$
determines a relative extension bifunctor $\mathbb{E}_\xi$ and a new
extriangulated structure $(\mathscr{C},\mathbb{E}_\xi,\mathfrak{s}_\xi)$ on
the same additive category. For exact categories, proper classes are
precisely the exact substructures (Lemma~\ref{lem:exact-substructures}),
such as the degreewise split sequences above. For triangulated categories
they include the pure triangles in a compactly generated triangulated
category \cite{Beligiannis2000,KrauseSmashing,KrauseCoherent} and, more
generally, the triangles which become short exact under a homological
functor; the resulting relative structures are in general neither exact
nor triangulated, so model structures relative to them are naturally
treated in the extriangulated setting.

This leads to the question addressed in this paper: can two admissible
model structures relative to different proper classes be mixed? Let
$(\mathscr{C},\mathbb{E},\mathfrak{s})$ be a weakly idempotent complete
extriangulated category and let $\xi_1\subseteq\xi_2$ be proper classes of
$\mathbb{E}$-triangles. Suppose that
\[
\mathcal{M}_1=(\mathcal{C}_1,\mathcal{W}_1,\mathcal{F}_1)
\quad\text{and}\quad
\mathcal{M}_2=(\mathcal{C}_2,\mathcal{W}_2,\mathcal{F}_2)
\]
are admissible model structures relative to $\xi_1$ and $\xi_2$,
respectively, with $\mathcal{W}_1\subseteq\mathcal{W}_2$; here
$\mathcal{C}_i$, $\mathcal{W}_i$ and $\mathcal{F}_i$ denote the classes of
cofibrant, trivial and fibrant objects. If
$\mathcal{F}_1\subseteq\mathcal{F}_2$, we prove that there is an admissible
model structure $\mathcal{M}_m=(\mathcal{C}_m,\mathcal{W}_2,\mathcal{F}_1)$
relative to $\xi_1$. Dually, if $\mathcal{C}_1\subseteq\mathcal{C}_2$, there
is an admissible model structure
$\mathcal{M}_m=(\mathcal{C}_1,\mathcal{W}_2,\mathcal{F}_m)$ relative to
$\xi_1$; see Theorem~\ref{prop:key_prop}.

The mixed model structure is admissible relative to $\xi_1$, but its
trivial objects are those of $\mathcal{M}_2$, which are defined through
$\xi_2$-triangles. The proof therefore has to work with the two relative
extension bifunctors $\mathbb{E}_{\xi_1}\subseteq\mathbb{E}_{\xi_2}$ at
once, that is, with two orthogonality relations and two theories of
cotorsion pairs on the same category. The main difficulty is to compare
the cotorsion pairs arising from the two theories and to prove
completeness of the mixed cotorsion pair. Lemma~\ref{lem:mixed-approximation}
provides a mixed $3\times3$ diagram in which $\xi_1$- and $\xi_2$-triangles
occur simultaneously, and Lemma~\ref{lem:extension-comparison} shows that
every $\xi_2$-triangle ending in an object of $\mathcal{C}_2$ is already a
$\xi_1$-triangle. Together, these two comparison steps bridge the two
orthogonality relations and form the technical core of the proof.

We also compare the three model structures at the level of morphisms; in
particular, $\mathcal{M}_m$ and $\mathcal{M}_2$ have the same weak
equivalences, and hence the same homotopy category. Moreover,
Proposition~\ref{prop:cofibrant-characterization} describes the mixed
cofibrant objects intrinsically: an object $C$ belongs to $\mathcal{C}_m$
if and only if $C\in\mathcal{C}_1$ and there is a $1$-weak equivalence
$C_2\to C$ with $C_2\in\mathcal{C}_2$; moreover, this weak equivalence can
be chosen to be a $\xi_2$-deflation. This describes $\mathcal{C}_m$ in terms
of the two original model structures, beyond its definition by relative
orthogonality.

Finally, we specialize to exact and triangulated categories. In the exact
case we recover the mixing theorem for exact model structures, including
the chain complex example above. In the triangulated case the theorem
applies to proper classes of triangles. We give a non-degenerate example
on the homotopy category $\mathrm{K}(R)$ of a quasi-Frobenius ring $R$,
with $\xi_1$ the class of cohomologically exact triangles and $\xi_2$ the
class of all triangles. There the mixed model structure has the chain maps
which are surjective in cohomology as fibrations and the chain maps with
K-injective mapping cone as weak equivalences; its cofibrant objects are
the complexes with projective cohomology, and it is not admissible with
respect to all triangles.

The paper is organized as follows. Section~\ref{sec:prelim} recalls
extriangulated categories, proper classes, relative cotorsion pairs and
admissible model structures. Section~\ref{sec:mixing} proves the main
mixing theorem and characterizes the mixed cofibrant objects.
Section~\ref{sec:applications} treats the exact and triangulated cases and
contains the examples.

\section{Preliminaries}\label{sec:prelim}

\subsection{Extriangulated categories and proper classes}
We refer to \cite{NP} for the definition of an extriangulated category
$(\mathscr{C},\mathbb{E},\mathfrak{s})$ and only fix notation. Here
$\mathscr{C}$ is an additive category,
$\mathbb{E}\colon\mathscr{C}^{\rm op}\times\mathscr{C}\to\mathsf{Ab}$ is a
biadditive functor, and $\mathfrak{s}$ is an additive realization of
$\mathbb{E}$, assigning to each $\delta\in\mathbb{E}(C,A)$ an equivalence
class $\mathfrak{s}(\delta)=[A\xrightarrow{x}B\xrightarrow{y}C]$ of
sequences, subject to the axioms of \cite[Definition 2.12]{NP}. For
morphisms $a\colon A\to A'$ and $c\colon C'\to C$ we write
$a_*\delta=\mathbb{E}(C,a)(\delta)$ and $c^*\delta=\mathbb{E}(c,A)(\delta)$.
If $\mathfrak{s}(\delta)=[A\xrightarrow{x}B\xrightarrow{y}C]$, then
$A\xrightarrow{x}B\xrightarrow{y}C\overset{\delta}{\dashrightarrow}$ is
called an \emph{$\mathbb{E}$-triangle}, $x$ an \emph{inflation} and $y$ a
\emph{deflation}, and we write $C={\rm Cone}(x)$ and $A={\rm CoCone}(y)$. An
$\mathbb{E}$-triangle is \emph{split} if $\delta=0$, equivalently if $x$ is
a split monomorphism, equivalently if $y$ is a split epimorphism. Exact
categories (with $\mathbb{E}=\Ext^1$) and triangulated categories (with
$\mathbb{E}(C,A)=\Hom(C,A[1])$) are extriangulated, and so are
extension-closed subcategories of triangulated categories \cite{NP}.

Proper classes of $\mathbb{E}$-triangles were introduced by Hu, Zhang and
Zhou \cite{hu2020proper}, following Beligiannis \cite{Beligiannis2000} in
the triangulated case. Their definition uses the following diagrams.

\begin{lem}\label{lem1} \emph{(see \cite[Proposition 3.15]{NP})} Let $(\mathscr{C}, \mathbb{E},\mathfrak{s})$ be an extriangulated category.
	
	\emph{(1)} Let $\xymatrix@C=2em{A_1\ar[r]^{x_1}&B_1\ar[r]^{y_1}&C\ar@{-->}[r]^{\delta_1}&}$ and $\xymatrix@C=2em{A_2\ar[r]^{x_2}&B_2\ar[r]^{y_2}&C\ar@{-->}[r]^{\delta_2}&}$ be $\mathbb{E}$-triangles. Then there is a commutative diagram
	$$\xymatrix{
		& A_2\ar[d]_{m_2} \ar@{=}[r] & A_2 \ar[d]^{x_2} \\
		A_1 \ar@{=}[d] \ar[r]^{m_1} & M \ar[d]_{e_2} \ar[r]^{e_1} & B_2\ar[d]^{y_2} \\
		A_1 \ar[r]^{x_1} & B_1\ar[r]^{y_1} & C   }
	$$
	with $\mathfrak{s}(y^*_2\delta_1)=\xymatrix@C=2em{[A_1\ar[r]^{m_1}&M\ar[r]^{e_1}&B_2]}$ and
	$\mathfrak{s}(y^*_1\delta_2)=\xymatrix@C=2em{[A_2\ar[r]^{m_2}&M\ar[r]^{e_2}&B_1].}$
	
	\emph{(2)} Let $\xymatrix@C=2em{A\ar[r]^{x_1}&B_1\ar[r]^{y_1}&C_1\ar@{-->}[r]^{\delta_1}&}$ and $\xymatrix@C=2em{A\ar[r]^{x_2}&B_2\ar[r]^{y_2}&C_2\ar@{-->}[r]^{\delta_2}&}$ be $\mathbb{E}$-triangles. Then there is a commutative diagram
	$$\xymatrix{
		A\ar[d]_{x_2} \ar[r]^{x_1} & B_1 \ar[d]^{m_2}\ar[r]^{y_1}&C_1\ar@{=}[d] \\
		B_2 \ar[d]_{y_2} \ar[r]^{m_1} & M \ar[d]^{e_2} \ar[r]^{e_1} & C_1 \\
		C_2 \ar@{=}[r] & C_2 &   }
	$$
	with $\mathfrak{s}(x_{2*}\delta_1)=\xymatrix@C=2em{[B_2\ar[r]^{m_1}&M\ar[r]^{e_1}&C_1]}$ and $\mathfrak{s}(x_{1*}\delta_2)=\xymatrix@C=2em{[B_1\ar[r]^{m_2}&M\ar[r]^{e_2}&C_2].}$
\end{lem}

Let $\xi$ be a class of $\mathbb{E}$-triangles which is closed under
isomorphisms. An $\mathbb{E}$-extension \emph{belongs to $\xi$} if it is
realized by an $\mathbb{E}$-triangle in $\xi$. The class $\xi$ is
\emph{closed under base change} (resp.\ \emph{cobase change}) if $c^*\delta$
(resp.\ $a_*\delta$) belongs to $\xi$ whenever $\delta$ does, for all
morphisms $c$ (resp.\ $a$); it is \emph{saturated} if, in the situation of
Lemma~\ref{lem1}(1), $\delta_1$ belongs to $\xi$ whenever $\delta_2$ and
$y_2^*\delta_1$ do. We denote by $\Delta_0$ the class of split
$\mathbb{E}$-triangles.

\begin{definition}[{\cite{hu2020proper}}]\label{def:proper class}
	A class $\xi$ of $\mathbb{E}$-triangles which is closed under isomorphisms
	is a \emph{proper class} if the following conditions hold:
	\begin{enumerate}
		\item[(1)] $\xi$ is closed under finite direct sums and $\Delta_0\subseteq\xi$;
		\item[(2)] $\xi$ is closed under base change and cobase change;
		\item[(3)] $\xi$ is saturated.
	\end{enumerate}
	In this case the $\mathbb{E}$-triangles in $\xi$ are called
	\emph{$\xi$-triangles}, and their inflations and deflations are called
	\emph{$\xi$-inflations} and \emph{$\xi$-deflations}.
\end{definition}

\begin{remark}[{\cite[Proposition 3.9]{hu2020proper}}]
	Let $\xi$ be a class of $\mathbb{E}$-triangles which is closed under
	isomorphisms, finite direct sums, base change and cobase change, and
	contains $\Delta_0$. Then $\xi$ is saturated if and only if, in the
	situation of Lemma~\ref{lem1}(2), $\delta_1$ belongs to $\xi$ whenever
	$\delta_2$ and $x_{2*}\delta_1$ do.
\end{remark}

\begin{thm} \emph{(see \cite[Theorem 3.2]{hu2020proper})}\label{thm:proper-extri}
	Let $\xi$ be a class of $\mathbb{E}$-triangles which is closed under
	isomorphisms. For $A,C\in\mathscr{C}$ let $\mathbb{E}_\xi(C, A)$ be the
	set of $\delta\in\mathbb{E}(C, A)$ belonging to $\xi$, and let
	$\mathfrak{s}_\xi$ be the restriction of $\mathfrak{s}$ to
	$\mathbb{E}_\xi$. Then $\xi$ is a proper class if and only if
	$(\mathscr{C}, \mathbb{E}_\xi, \mathfrak{s}_\xi)$ is an extriangulated
	category.
\end{thm}

In this case $\mathbb{E}_\xi$ is a subfunctor of $\mathbb{E}$, and the
conflations of $(\mathscr{C},\mathbb{E}_\xi,\mathfrak{s}_\xi)$ are the
$\xi$-triangles. For $\xi$ the class of all $\mathbb{E}$-triangles we
recover $(\mathscr{C},\mathbb{E},\mathfrak{s})$.

\begin{lem}\label{lem:3x3}
	Let $(\mathscr{C},\mathbb{E},\mathfrak{s})$ be an extriangulated category, and let
	\[
	X\xrightarrow{j}Y\xrightarrow{g}L\overset{\delta_j}{\dashrightarrow},\qquad
	K\xrightarrow{k}Y\xrightarrow{q}Z\dashrightarrow,\qquad
	K'\xrightarrow{f}X\xrightarrow{p}Z\dashrightarrow
	\]
	be $\mathbb{E}$-triangles with $p=qj$. Then there is an $\mathbb{E}$-triangle
	$K'\xrightarrow{a}K\xrightarrow{gk}L\dashrightarrow$ with $ka=jf$.
	If $\xi$ is a proper class containing the three given triangles,
	then this triangle can be chosen in $\xi$.
\end{lem}

\begin{proof}
	This is \cite[Proposition~3.17]{NP}.
	The assertion for $\xi$-triangles follows from
	Theorem~\ref{thm:proper-extri}.
\end{proof}

\subsection{Cotorsion pairs relative to a proper class}

Let $\xi$ be a proper class of $\mathbb{E}$-triangles. For classes
$\mathcal{X}$ and $\mathcal{Y}$ of objects of $\mathscr{C}$ put
\[
\begin{aligned}
	\mathcal{X}^{\perp_\xi}&=\{Z\in\mathscr{C}\mid \mathbb{E}_{\xi}(X, Z)=0 \text{ for all } X\in\mathcal{X}\},\\
	{}^{\perp_\xi}\mathcal{Y}&=\{Z\in\mathscr{C}\mid \mathbb{E}_{\xi}(Z, Y)=0 \text{ for all } Y\in\mathcal{Y}\}.
\end{aligned}
\]

\begin{definition}
	Let $\mathcal{X},\mathcal{Y}\subseteq\mathscr{C}$ be full additive
	subcategories which are closed under isomorphisms and direct summands.
	The pair $(\mathcal{X},\mathcal{Y})$ is a \emph{cotorsion pair relative to
		$\xi$} if $\mathcal{X}^{\perp_\xi}=\mathcal{Y}$ and
	${}^{\perp_\xi}\mathcal{Y}=\mathcal{X}$. It is \emph{complete} if every
	$C\in\mathscr{C}$ admits $\xi$-triangles
	$$\xymatrix{Y_C\ar[r]&X_C\ar[r]&C\ar@{-->}[r]&}
	\qquad\text{and}\qquad
	\xymatrix{C\ar[r]&Y^C\ar[r]&X^C\ar@{-->}[r]&}$$
	with $X_C,X^C\in\mathcal{X}$ and $Y_C,Y^C\in\mathcal{Y}$, and
	\emph{hereditary} if $\mathcal{X}$ is closed under cocones of
	$\xi$-deflations and $\mathcal{Y}$ is closed under cones of
	$\xi$-inflations. When $\xi$ is the class of all $\mathbb{E}$-triangles,
	we omit the phrase ``relative to $\xi$''.
\end{definition}

\subsection{Admissible model structures and Hovey triples}

We use model structures in the sense of Quillen \cite{Quillen_1967}: a
\emph{model structure} on a category $\mathscr{M}$ is a triple
$(\mathrm{CoFib},\mathrm{Fib},\mathrm{Weq})$ of classes of morphisms,
called cofibrations, fibrations and weak equivalences, satisfying the
two-out-of-three, retract, lifting and factorization axioms. We put
$\mathrm{TCoFib}=\mathrm{CoFib}\cap\mathrm{Weq}$ and
$\mathrm{TFib}=\mathrm{Fib}\cap\mathrm{Weq}$. If $\mathscr{M}$ has a zero
object, an object $X$ is \emph{cofibrant} if $0\to X$ is a cofibration,
\emph{fibrant} if $X\to0$ is a fibration, and \emph{trivial} if $X\to0$,
or equivalently $0\to X$, is a weak equivalence. We denote by
$\mathcal{C}$, $\mathcal{F}$ and $\mathcal{W}$ the classes of cofibrant,
fibrant and trivial objects, respectively.

\begin{definition}[{\cite[Definition 5.5]{NP}}]\label{def:admissible} 
	A model structure $(\mathrm{CoFib}, \mathrm{Fib}, \mathrm{Weq})$ on an extriangulated category $(\mathscr{C},\mathbb{E}, \mathfrak{s})$ is \emph{admissible} if
	\begin{align*}
		\mathrm{CoFib}&=\{ \text{inflations } f \mid \mathrm{Cone}(f)\in \mathcal{C} \}, &
		\mathrm{TCoFib}&=\{ \text{inflations } f \mid \mathrm{Cone}(f)\in \mathcal{C}\cap \mathcal{W} \},\\
		\mathrm{Fib}&=\{ \text{deflations } f \mid \mathrm{CoCone}(f)\in \mathcal{F} \}, &
		\mathrm{TFib}&=\{ \text{deflations } f \mid \mathrm{CoCone}(f)\in \mathcal{F}\cap \mathcal{W} \}.
	\end{align*}
	For a proper class $\xi$, a model structure on $\mathscr{C}$ is
	\emph{admissible relative to $\xi$} if it is an admissible model
	structure on $(\mathscr{C},\mathbb{E}_\xi,\mathfrak{s}_\xi)$, that is,
	if the above holds with inflations and deflations replaced by
	$\xi$-inflations and $\xi$-deflations.
\end{definition}

\begin{definition}
	A class $\mathcal{W}$ of objects of $\mathscr{C}$ is \emph{thick} if it
	is closed under direct summands and has the \emph{two-out-of-three
		property}: whenever two of the three terms of an $\mathbb{E}$-triangle
	lie in $\mathcal{W}$, so does the third. A triple
	$(\mathcal{C}, \mathcal{W}, \mathcal{F})$ of classes of objects is a
	\emph{Hovey triple} if $(\mathcal{C}\cap \mathcal{W}, \mathcal{F})$ and
	$(\mathcal{C}, \mathcal{F}\cap \mathcal{W})$ are complete cotorsion pairs
	and $\mathcal{W}$ is thick; it is \emph{hereditary} if both cotorsion
	pairs are hereditary. For a proper class $\xi$, (hereditary) Hovey triples
	\emph{relative to $\xi$} are defined in the same way, using cotorsion
	pairs relative to $\xi$ and the two-out-of-three property for
	$\xi$-triangles.
\end{definition}

\begin{lem}[{\cite[Condition 5.8]{NP}, \cite[Proposition 2.7]{Klapproth2026}}]\label{lem:wic}
	For an extriangulated category $(\mathscr{C},\mathbb{E}, \mathfrak{s})$
	the following are equivalent: \emph{(1)} every split monomorphism has a
	cokernel; \emph{(2)} every split epimorphism has a kernel; \emph{(3)} if
	$ki$ is an inflation, then so is $i$; \emph{(4)} if $de$ is a deflation,
	then so is $d$.
\end{lem}    

An extriangulated category satisfying these conditions is called
\emph{weakly idempotent complete}.

\begin{prop}\label{prop:wic-relative}
	If $(\mathscr{C},\mathbb{E},\mathfrak{s})$ is weakly idempotent complete
	and $\xi$ is a proper class of $\mathbb{E}$-triangles, then
	$(\mathscr{C},\mathbb{E}_{\xi},\mathfrak{s}_{\xi})$ is weakly idempotent
	complete.
\end{prop}

\begin{proof}
	We verify condition (3) of Lemma~\ref{lem:wic}. Let $i\colon X\to Y$ and
	$k\colon Y\to Z$ be such that $ki$ is a $\xi$-inflation, with
	$\xi$-triangle $X \xrightarrow{ki} Z \to C\overset{\delta}{\dashrightarrow}$.
	Since $(\mathscr{C},\mathbb{E},\mathfrak{s})$ is weakly idempotent
	complete, $i$ is an inflation; let
	$X \xrightarrow{i} Y \to C' \overset{\eta}{\dashrightarrow}$ be an
	$\mathbb{E}$-triangle. As $k\circ i=(ki)\circ 1_X$, axiom (ET3) yields a
	morphism $c\colon C'\to C$ with $\eta=c^*\delta$. Since $\xi$ is closed
	under base change, $\eta$ belongs to $\xi$, so $i$ is a $\xi$-inflation.
\end{proof}

\begin{thm}[{\cite[Section 5]{NP}}]\label{thm:relative-hovey}
	Let $(\mathscr{C},\mathbb{E},\mathfrak{s})$ be a weakly idempotent
	complete extriangulated category and let $\xi$ be a proper class of
	$\mathbb{E}$-triangles. Then
	$(\mathrm{CoFib}, \mathrm{Fib}, \mathrm{Weq})\mapsto (\mathcal{C}, \mathcal{W}, \mathcal{F})$
	is a one-to-one correspondence between admissible model structures
	relative to $\xi$ and Hovey triples relative to $\xi$. The inverse sends
	$(\mathcal{C}, \mathcal{W}, \mathcal{F})$ to
	\begin{align*}
		\mathrm{CoFib}_\xi&=\{ \xi\text{-inflations } f \mid \mathrm{Cone}(f)\in \mathcal{C} \},\qquad
		\mathrm{Fib}_\xi=\{ \xi\text{-deflations } f \mid \mathrm{CoCone}(f)\in \mathcal{F} \},\\
		\mathrm{Weq}_\xi&=\{pi \mid i \text{ a } \xi\text{-inflation with } \mathrm{Cone}(i)\in \mathcal{C}\cap \mathcal{W},\\
		&\qquad\quad\ \, p \text{ a } \xi\text{-deflation with } \mathrm{CoCone}(p)\in \mathcal{F}\cap \mathcal{W}\}.
	\end{align*}
\end{thm}

\begin{proof}
	For the class of all $\mathbb{E}$-triangles this is the Hovey
	correspondence of Nakaoka and Palu \cite[Section 5]{NP}. In general,
	$(\mathscr{C},\mathbb{E}_\xi,\mathfrak{s}_\xi)$ is a weakly idempotent
	complete extriangulated category by Theorem~\ref{thm:proper-extri} and
	Proposition~\ref{prop:wic-relative}, whose conflations are the
	$\xi$-triangles; its admissible model structures and Hovey triples are
	precisely those relative to $\xi$, so the absolute case applies.
\end{proof}

By Theorem~\ref{thm:relative-hovey}, an admissible model structure relative
to $\xi$ is determined by its Hovey triple. We therefore identify the two,
write $\mathcal{M}=(\mathcal{C},\mathcal{W},\mathcal{F})$, and call
$\mathcal{M}$ \emph{hereditary} if its Hovey triple is.

\section{Mixing model structures on extriangulated categories}\label{sec:mixing}

\subsection{Mixed model structure}

Let $(\mathscr{C},\mathbb{E}, \mathfrak{s})$ be an extriangulated category, and $\xi_1$ a proper class of $\mathbb{E}$-triangles. Suppose that $\mathcal{M}_1= (\mathcal{C}_1, \mathcal{W}_1, \mathcal{F}_1)$ is an admissible model structure relative to $\xi_1$. Moreover, assume $\xi_2$ is another proper class on $\mathscr{C}$, and that we have another admissible model structure $\mathcal{M}_2= (\mathcal{C}_2, \mathcal{W}_2, \mathcal{F}_2)$ relative to $\xi_2$.

To distinguish between the two admissible model structures, we adopt the following notation throughout this section. The relative extension groups
\[
\mathbb{E}_{\xi_1}(A,B) \quad \text{and} \quad \mathbb{E}_{\xi_2}(A,B)
\]
will be denoted simply by
\[
 \mathbb{E}_{1}(A,B) \quad \text{and} \quad  \mathbb{E}_{2}(A,B),
\]
respectively. 

We shall also use terminology such as $\xi_1$-cofibrations, $\xi_2$-fibrations, $\xi_1$-trivial cofibrations, $\xi_2$-weak equivalences, and so on. Finally, to distinguish the corresponding orthogonal classes, we write
\[
{}^{\perp_1}\mathcal{C}, \quad \mathcal{C}^{\perp_1}, \quad {}^{\perp_2}\mathcal{C}, \quad \mathcal{C}^{\perp_2},
\]
for the left and right orthogonal classes with respect to $\mathbb{E}_{\xi_1}(-,-)$ and $\mathbb{E}_{\xi_2}(-,-)$, respectively.

\begin{thm}[Cole's Theorem for extriangulated model structures]\label{prop:key_prop}
	Let $(\mathscr{C},\mathbb{E}, \mathfrak{s})$ be a weakly idempotent complete extriangulated category with two proper classes of $\mathbb{E}$-triangles $\xi_1$ and $\xi_2$ satisfying $\xi_1\subseteq \xi_2$. Assume that $\mathcal{M}_1=(\mathcal{C}_1,\mathcal{W}_1,\mathcal{F}_1)$ is an admissible model structure relative to $\xi_1$, and $\mathcal{M}_2=(\mathcal{C}_2,\mathcal{W}_2,\mathcal{F}_2)$ is an admissible model structure relative to $\xi_2$, such that $\mathcal{W}_1\subseteq \mathcal{W}_2$.
	
	\begin{enumerate}
		\item If $\mathcal{F}_1\subseteq \mathcal{F}_2$, then there exists an admissible model structure $\mathcal{M}_m=(\mathcal{C}_m,\mathcal{W}_2,\mathcal{F}_1)$ relative to $\xi_1$, having the same trivial objects as $\mathcal{M}_2$ and the same fibrant objects as $\mathcal{M}_1$.
		
		\item If $\mathcal{C}_1\subseteq \mathcal{C}_2$, then there exists an admissible model structure $\mathcal{M}_m=(\mathcal{C}_1,\mathcal{W}_2,\mathcal{F}_m)$ relative to $\xi_1$, having the same trivial objects as $\mathcal{M}_2$ and the same cofibrant objects as $\mathcal{M}_1$.
	\end{enumerate}
	
	In either case, we call $\mathcal{M}_m$ the \emph{mixed model structure} of $\mathcal{M}_1$ and $\mathcal{M}_2$.
\end{thm}

Now, in view of the mixed model structure we wish to construct, the following definition is naturally suggested.

\begin{definition}
	Define the class of objects
	\[ \mathcal{C}_m:={}^{\perp_1}(\mathcal{W}_2\cap \mathcal{F}_1). \]
	We call $\mathcal{C}_m$ the class of \emph{mixed cofibrant objects}, or simply the class of \emph{$m$-cofibrant objects}.
\end{definition}

It follows immediately from the definition that $\mathcal{C}_2\subseteq \mathcal{C}_m\subseteq \mathcal{C}_1$. We prove this explicitly in the next lemma.

\begin{lem}\label{lem:containments}
	Under the assumptions of Theorem \ref{prop:key_prop} (in particular, $\xi_1 \subseteq \xi_2$, $\mathcal{W}_1\subseteq \mathcal{W}_2$, and $\mathcal{F}_1\subseteq \mathcal{F}_2$), the following containments hold:
	\begin{enumerate}
		\item[(1)] For any class of objects $\mathcal{X}$, we have
		\[ {}^{\perp_2}\mathcal{X}\subseteq {}^{\perp_1}\mathcal{X} \qquad \text{and} \qquad \mathcal{X}^{\perp_2}\subseteq \mathcal{X}^{\perp_1}. \]
		\item[(2)] $\mathcal{C}_2\subseteq \mathcal{C}_m\subseteq \mathcal{C}_1$.
	\end{enumerate}
\end{lem}

\begin{proof}
	(1) Since $\xi_1 \subseteq \xi_2$, any $\xi_1$-extension is a $\xi_2$-extension. Thus, $\mathbb{E}_{1}(X,C)$ is a subgroup of $\mathbb{E}_{2}(X,C)$. It naturally follows that $\mathbb{E}_{2}(X,C)=0$ implies $\mathbb{E}_{1}(X,C)=0$. Hence, ${}^{\perp_2}\mathcal{X}\subseteq {}^{\perp_1}\mathcal{X}$. Dually, we have $\mathcal{X}^{\perp_2}\subseteq \mathcal{X}^{\perp_1}$.
	
	(2) Since $\mathcal{F}_1\subseteq \mathcal{F}_2$, we have $\mathcal{W}_2\cap\mathcal{F}_1 \subseteq \mathcal{W}_2\cap\mathcal{F}_2$. Taking the left orthogonal class with respect to $\mathbb{E}_{1}(-,-)$ reverses the inclusion, yielding
	\[ {}^{\perp_1}(\mathcal{W}_2\cap\mathcal{F}_2) \subseteq {}^{\perp_1}(\mathcal{W}_2\cap\mathcal{F}_1). \]
	By applying part (1), we deduce the first sequence of containments:
	\begin{align*}
		\mathcal{C}_2 &= {}^{\perp_2}(\mathcal{W}_2\cap\mathcal{F}_2) \\
		&\subseteq {}^{\perp_1}(\mathcal{W}_2\cap\mathcal{F}_2) \\
		&\subseteq {}^{\perp_1}(\mathcal{W}_2\cap\mathcal{F}_1) = \mathcal{C}_m.
	\end{align*}
	
	On the other hand, since $\mathcal{W}_1 \subseteq \mathcal{W}_2$, we have $\mathcal{W}_1\cap\mathcal{F}_1 \subseteq \mathcal{W}_2\cap\mathcal{F}_1$. Reversing the inclusion again via the left orthogonal class gives:
	\begin{align*}
		\mathcal{C}_m &= {}^{\perp_1}(\mathcal{W}_2\cap\mathcal{F}_1) \\
		&\subseteq {}^{\perp_1}(\mathcal{W}_1\cap\mathcal{F}_1) = \mathcal{C}_1.
	\end{align*}
	This completes the proof.
\end{proof}

The following lemma shows that the trivial cofibrations in the mixed extriangulated model structure \(\mathcal{M}_m\) coincide with the trivial cofibrations in \(\mathcal{M}_1\).

\begin{lem}\label{lem:Cm-cap-W2}
	$
	\mathcal{C}_m\cap \mathcal{W}_2
	=
	\mathcal{C}_1\cap \mathcal{W}_1.
	$
\end{lem}
\begin{proof}
	(\(\subseteq\))
	Let \(C\in \mathcal{C}_m\cap\mathcal{W}_2\).
	Since \((\mathcal{C}_1\cap\mathcal{W}_1,\mathcal{F}_1)\) has enough projectives, there exists a $\xi_{1}$-triangle
$$\xymatrix@C=2em{A\ar[r]&B\ar[r]&C\ar@{-->}[r]^{\delta}&}$$
	with
	\(
	A\in\mathcal{F}_1
	\)
	and
	\(
	B\in\mathcal{C}_1\cap\mathcal{W}_1.
	\)
	
	Since
	\(
	\mathcal{W}_1\subseteq\mathcal{W}_2
	\)
	and
	\(
	\xi_1\subseteq\xi_2,
	\)
	and since \(\mathcal{W}_2\) satisfies the two-out-of-three property with respect to $\xi_2$-triangles, we obtain
	\(
	A\in\mathcal{W}_2\cap\mathcal{F}_1.
	\)
	
	By the definition
	\(
	\mathcal{C}_m={}^{\perp_1}(\mathcal{W}_2\cap\mathcal{F}_1),
	\)
	the above triangle splits. Hence \(C\) is a direct summand of \(B\). Therefore,
	\(
	C\in\mathcal{C}_1\cap\mathcal{W}_1.
	\)
	
	(\(\supseteq\))
	Since
	\(
	\mathcal{W}_2\cap\mathcal{F}_1\subseteq\mathcal{F}_1,
	\)
	we have
	\(
	{}^{\perp_1}\mathcal{F}_1
	\subseteq
	{}^{\perp_1}(\mathcal{W}_2\cap\mathcal{F}_1).
	\)
	That is,
	\(
	\mathcal{C}_1\cap\mathcal{W}_1
	\subseteq
	\mathcal{C}_m.
	\)
	
	Moreover,
	\(
	\mathcal{C}_1\cap\mathcal{W}_1
	\subseteq
	\mathcal{W}_1
	\subseteq
	\mathcal{W}_2.
	\)
	Hence,
	\(
	\mathcal{C}_1\cap\mathcal{W}_1
	\subseteq
	\mathcal{C}_m\cap\mathcal{W}_2.
	\)
\end{proof}

The following lemma will be used to show that one cotorsion pair of the corresponding Hovey triple in the mixed model structure $\mathcal{M}_m=(\mathcal{C}_m,\mathcal{W}_2,\mathcal{F}_1)$ is complete. It will also help us characterize the class of $m$-cofibrant objects later.

\begin{lem}\label{lem:mixed-approximation}
	For any object $A\in\mathscr{C}$, there exists a commutative diagram
	\[
	\xymatrix{
		F_2 \ar[r] \ar[d] & C_2 \ar[r]^{p} \ar[d]^{j} & A \ar@{=}[d] \\
		F_m \ar[r] \ar[d] & C_m \ar[r]^{q} \ar[d] & A \\
		C_1 \ar@{=}[r] & C_1 &
	}
	\]
	such that the second row and the second column are $\xi_{1}$-triangles, while the first row and the first column are $\xi_{2}$-triangles. Moreover,
	\[ C_2\in\mathcal{C}_2, \qquad F_2\in\mathcal{W}_2\cap\mathcal{F}_2, \]
	\[ C_m\in\mathcal{C}_m, \qquad F_m\in\mathcal{W}_2\cap\mathcal{F}_1, \]
	and $C_1 \in \mathcal{C}_m\cap\mathcal{W}_2 = \mathcal{C}_1\cap\mathcal{W}_1$.
\end{lem}

\begin{proof}
	Since $(\mathcal{C}_2,\mathcal{W}_2\cap\mathcal{F}_2)$ is complete, there is a
	$\xi_2$-triangle $F_2\to C_2\xrightarrow{p}A\dashrightarrow$ with
	$C_2\in\mathcal{C}_2$ and $F_2\in\mathcal{W}_2\cap\mathcal{F}_2$. Factor
	$p=qj$ in $\mathcal{M}_1$ with $j$ a $\xi_1$-trivial cofibration and $q$ a
	$\xi_1$-fibration. This gives the second column
	$C_2\xrightarrow{j}C_m\to C_1\dashrightarrow$ and the second row
	$F_m\to C_m\xrightarrow{q}A\dashrightarrow$, both $\xi_1$-triangles, with
	$C_1\in\mathcal{C}_1\cap\mathcal{W}_1=\mathcal{C}_m\cap\mathcal{W}_2$ by
	Lemma~\ref{lem:Cm-cap-W2}, and $F_m\in\mathcal{F}_1$. As
	$\mathcal{C}_m={}^{\perp_1}(\mathcal{W}_2\cap\mathcal{F}_1)$ is closed under
	$\xi_1$-extensions and contains $C_2$ and $C_1$, we have $C_m\in\mathcal{C}_m$.

	The first row, the second row and the second column are $\xi_2$-triangles,
	so Lemma~\ref{lem:3x3} for $\xi_2$ provides the first column, a
	$\xi_2$-triangle $F_2\to F_m\to C_1\dashrightarrow$ making the diagram
	commute. Since $F_2,C_1\in\mathcal{W}_2$ and $\mathcal{W}_2$ is thick
	relative to $\xi_2$, we get $F_m\in\mathcal{W}_2\cap\mathcal{F}_1$.
\end{proof}

\begin{lem}\label{lem:extension-comparison}
	Any $\xi_2$-triangle

\[
\xymatrix@C=2em{
A \ar[r] &
	B \ar[r] &
	C\ar@{-->}[r] &
	{}
}
\]
	with $C\in\mathcal{C}_2$ is necessarily a $\xi_1$-triangle. Consequently,
	$
	\mathbb{E}_{2}(C,X)
	=
	\mathbb{E}_{1}(C,X)
	$
	whenever $C\in\mathcal{C}_2$.
\end{lem}
\begin{proof}Consider the cobase change diagram
	\begin{displaymath}
		\xymatrix{A \ar[r]\ar[d] & B \ar[r]\ar[d]^{} & C \ar@{=}[d] & \\
			R \ar[r]^{} \ar[d] & P \ar[r] \ar[d] & C  &  \\
			Q \ar@{=}[r]  & Q  &  & \\
			 }
	\end{displaymath}
	where the first column arises from the complete cotorsion pair
	$(\mathcal{C}_1,\mathcal{W}_1 \cap \mathcal{F}_1)$ with respect to $\xi_1$.
	
	Note that $R \in \mathcal{W}_1 \cap \mathcal{F}_1 \subseteq \mathcal{W}_2 \cap \mathcal{F}_2$. Moreover, the second row also lies in $\xi_2$, hence it is split. In particular, it belongs to $\xi_1$.
	Therefore, by the dual saturation criterion for proper classes
	\cite[Proposition 3.9]{hu2020proper}, we conclude that
\[
\xymatrix@C=2em{
	A \ar[r] &
	B \ar[r] &
	C\ar@{-->}[r] &
	{}
}
\]
	is in $\xi_1$.
\end{proof}
\begin{cor}\label{cor:mixed-orthogonality}
		\[
	(\mathcal{C}_m)^{\perp_1}=\mathcal{W}_2\cap \mathcal{F}_1.
	\]
\end{cor}
\begin{proof}
	By Lemmas~\ref{lem:containments} and~\ref{lem:extension-comparison},
	\[
	(\mathcal{C}_m)^{\perp_1}\subseteq(\mathcal{C}_2)^{\perp_1}
	=(\mathcal{C}_2)^{\perp_2}=\mathcal{W}_2\cap\mathcal{F}_2.
	\]
	Moreover, Lemma~\ref{lem:Cm-cap-W2} gives
	$\mathcal{C}_1\cap\mathcal{W}_1\subseteq\mathcal{C}_m$, so
	$(\mathcal{C}_m)^{\perp_1}\subseteq
	(\mathcal{C}_1\cap\mathcal{W}_1)^{\perp_1}=\mathcal{F}_1$.
	Combining these containments yields
	$(\mathcal{C}_m)^{\perp_1}\subseteq\mathcal{W}_2\cap\mathcal{F}_1$.
	The reverse inclusion follows immediately from the definition of
	$\mathcal{C}_m$.
\end{proof}
With the tools established in the previous lemmas, we are now ready to prove the existence of the mixed model structure.
\subsection{Proof of Theorem~\ref{prop:key_prop}}

We use the aforementioned results to prove (1). Statement (2) follows from duality.

\begin{proof}
	We first show that \(\mathcal{W}_2\) is thick with respect to \(\xi_1\)-triangles. Indeed, let
\[
\xymatrix@C=2em{
	A \ar[r] &
	B \ar[r] &
	C\ar@{-->}[r] &
	{}
}
\]
	be a \(\xi_1\)-triangle. Since \(\xi_1\subseteq\xi_2\), this is also a \(\xi_2\)-triangle. As \(\mathcal{W}_2\) is thick with respect to \(\xi_2\), whenever two of \(A,B,C\) belong to \(\mathcal{W}_2\), so does the third.
	
	Next, by Lemma~\ref{lem:Cm-cap-W2},
	\[
	(\mathcal{C}_m\cap\mathcal{W}_2,\mathcal{F}_1)
	=
	(\mathcal{C}_1\cap\mathcal{W}_1,\mathcal{F}_1),
	\]
	which is a complete cotorsion pair with respect to $\xi_1$.
	
	On the other hand, by Corollary~\ref{cor:mixed-orthogonality},
	\[
	(\mathcal{C}_m,\mathcal{F}_1\cap\mathcal{W}_2)
	\]
	is a cotorsion pair with respect to $\xi_1$. It remains to prove that it is complete.
	
	By Lemma~\ref{lem:mixed-approximation}, for every object \(A\in\mathscr{C}\), there exists a \(\xi_1\)-triangle
	\[
	\xymatrix@C=2em{
		F_m \ar[r] &
		C_m \ar[r] &
		A\ar@{-->}[r] &
		{}
	}
	\]
	with
	\(
	C_m\in\mathcal{C}_m
	\)
	and
	\(
	F_m\in\mathcal{W}_2\cap\mathcal{F}_1.
	\)
	
	Thus it suffices to show that for every object \(A\in\mathscr{C}\), there exists a \(\xi_1\)-triangle
	\[
\xymatrix@C=2em{
	A \ar[r] &
	P \ar[r] &
	C_m\ar@{-->}[r] &
	{}
}
\]
	with
	\(
	P\in\mathcal{W}_2\cap\mathcal{F}_1
	\)
	and
	\(
	C_m\in\mathcal{C}_m.
	\)
	
	Since
	\(
	(\mathcal{C}_1,\mathcal{F}_1\cap\mathcal{W}_1)
	\)
	is a complete cotorsion pair with respect to $\xi_1$, there exists a \(\xi_1\)-triangle
	\[
	\xymatrix@C=2em{
		A \ar[r] &
		F_1 \ar[r] &
		C_1\ar@{-->}[r] &
		{}
	}
	\]
	with
	\(
	F_1\in\mathcal{F}_1\cap\mathcal{W}_1
	\)
	and
	\(
	C_1\in\mathcal{C}_1.
	\)
	
	Applying the first half of completeness to \(C_1\), there exists a \(\xi_1\)-triangle
	\[
	\xymatrix@C=2em{
		F_m \ar[r] &
		C_m \ar[r] &
		C_1\ar@{-->}[r] &
		{}
	}
	\]
	with
	\(
	F_m\in\mathcal{W}_2\cap\mathcal{F}_1
	\)
	and
	\(
	C_m\in\mathcal{C}_m.
	\)
	
Now consider the following base change diagram:
\[
\xymatrix{
	& F_m \ar@{=}[r] \ar[d]
	& F_m \ar[d]
	&
	\\
	A \ar[r] \ar@{=}[d]
	& P \ar[r] \ar[d]
	& C_m  \ar[d]
	& 
	\\
	A \ar[r]
	& F_1 \ar[r] 
	& C_1  
	& 
	\\
	& 
	& 
	&
}
\]

Since \(\xi_1\) is closed under base change, the second row and the second column are both \(\xi_1\)-triangles. Since
\(
F_1,F_m\in\mathcal{W}_2\cap\mathcal{F}_1,
\)
it follows that
\(
P\in\mathcal{W}_2\cap\mathcal{F}_1.
\)
Hence the second row is the desired \(\xi_1\)-triangle.

\end{proof}

\begin{cor}\label{cor:hereditary}
	Under the assumptions of Theorem~\ref{prop:key_prop}(1), if $\mathcal{M}_1$ is hereditary, then so is $\mathcal{M}_m$.
\end{cor}

\begin{proof}
	By Lemma~\ref{lem:Cm-cap-W2},
	$(\mathcal{C}_m\cap\mathcal{W}_2,\mathcal{F}_1)=(\mathcal{C}_1\cap\mathcal{W}_1,\mathcal{F}_1)$,
	which is hereditary.

	Let $X\to Y\to Z\dashrightarrow$ be a $\xi_1$-triangle with
	$X,Y\in\mathcal{W}_2\cap\mathcal{F}_1$. It is a $\xi_2$-triangle, so
	$Z\in\mathcal{W}_2$; and $Z\in\mathcal{F}_1$ as $\mathcal{M}_1$ is
	hereditary. Hence $\mathcal{W}_2\cap\mathcal{F}_1$ is closed under cones of
	$\xi_1$-inflations.

	Let $K\xrightarrow{u}X_0\to X_1\dashrightarrow$ be a $\xi_1$-triangle with
	$X_0,X_1\in\mathcal{C}_m$, and let $F\in\mathcal{W}_2\cap\mathcal{F}_1$. As
	$\mathcal{C}_m\subseteq\mathcal{C}_1$ and $\mathcal{M}_1$ is hereditary,
	$K\in\mathcal{C}_1$. Completeness of $(\mathcal{C}_1,\mathcal{W}_1\cap\mathcal{F}_1)$
	gives a $\xi_1$-triangle $F\to J\xrightarrow{v}Q\dashrightarrow$ with
	$J\in\mathcal{W}_1\cap\mathcal{F}_1\subseteq\mathcal{W}_2\cap\mathcal{F}_1$,
	so $Q\in\mathcal{W}_2\cap\mathcal{F}_1$ by the previous paragraph. Since
	$\mathbb{E}_1(X_1,Q)=0=\mathbb{E}_1(X_0,F)$, every morphism $K\to Q$ factors
	through $u$, and every morphism $X_0\to Q$ factors through $v$. Hence
	$\Hom(K,J)\to\Hom(K,Q)$ is surjective, and the exact sequence
	\[
	\Hom(K,J)\to\Hom(K,Q)\to\mathbb{E}_1(K,F)\to\mathbb{E}_1(K,J)=0
	\]
	gives $\mathbb{E}_1(K,F)=0$. Thus $K\in\mathcal{C}_m$.
\end{proof}

\subsection{Cofibrant Objects in Mixed Extriangulated Model Structures}

Having established the existence of the mixed model structure $\mathcal{M}_m$, we now turn our attention to explicitly characterizing its cofibrant objects. As a first step, we compare the classes of morphisms in $\mathcal{M}_m$ with those in $\mathcal{M}_1$ and $\mathcal{M}_2$.

\begin{prop}\label{prop:morphism-comparison}
	Let $(\mathscr{C},\mathbb{E},\mathfrak{s})$ be a weakly idempotent complete extriangulated category, and let $\mathcal{M}_m=(\mathcal{C}_m,\mathcal{W}_2,\mathcal{F}_1)$ be the mixed model structure relative to $\xi_1$ constructed above, and let $f$ be a morphism in $\mathscr{C}$.
	We have the following relationships between the classes of morphisms.
	\begin{enumerate}
		\item
		$f$ is an $m$-fibration if and only if $f$ is a $1$-fibration.
		
		\item
		$f$ is an $m$-trivial cofibration if and only if $f$ is a $1$-trivial cofibration.
		
		\item
		If $f$ is a $2$-cofibration, then $f$ is an $m$-cofibration. Moreover, if $f$ is an $m$-cofibration, then $f$ is a $1$-cofibration.
		
		\item
		If $f$ is a $1$-trivial fibration, then $f$ is an $m$-trivial fibration. Moreover, if $f$ is an $m$-trivial fibration, then $f$ is a $2$-trivial fibration.
		
		\item
		The class of $m$-weak equivalences coincides with the class of $2$-weak equivalences, and contains the class of all $1$-weak equivalences.
	\end{enumerate}
\end{prop}

\begin{proof}
	The first three statements follow immediately from the definitions. Moreover, (4) follows directly from (3). Thus it remains to prove (5).
	
	We first show that every $m$-weak equivalence is a $2$-weak equivalence. Let
	\(
	f:X\to Y
	\)
	be an $m$-weak equivalence. Consider a factorization of \(f\) in the mixed model structure \(\mathcal{M}_m\):
	\[
	\begin{tikzcd}[column sep=small]
		X \arrow[rr,"f"] \arrow[dr,"i"']
		&& Y \\
		& Z \arrow[ur,"p"'] &
	\end{tikzcd}
	\]
	where

	\[
	\xymatrix@C=2em{
		X  \ar[r]^i &
		Z \ar[r] &
		L\ar@{-->}[r] &
		{}
	}
	\]
	and

	\[
	\xymatrix@C=2em{
		K \ar[r] &
		Z \ar[r]^p &
		 Y\ar@{-->}[r] &
		{}
	}
	\]
	are \(\xi_1\)-triangles with
	\(
	L\in\mathcal{C}_m\cap\mathcal{W}_2
	\)
	and
	\(
	K\in\mathcal{F}_1\cap\mathcal{W}_2.
	\)
	
	Since \(L\in\mathcal{W}_2\), the morphism \(i\) is a $2$-weak equivalence. Similarly, \(p\) is also a $2$-weak equivalence. Since \(\mathrm{Weq}_2\) satisfies the two-out-of-three property, it follows that \(f\) is a $2$-weak equivalence.
	
	Conversely, let
	\(
	f:X\to Y
	\)
	be a $2$-weak equivalence. Consider a factorization of \(f\) in \(\mathcal{M}_m\):
	\[
	\begin{tikzcd}[column sep=small]
		X \arrow[rr,"f"] \arrow[dr,"i"']
		&& Y \\
		& Z \arrow[ur,"p"'] &
	\end{tikzcd}
	\]
	where

\[
\xymatrix@C=2em{
	X  \ar[r]^i &
	Z \ar[r] &
	L\ar@{-->}[r] &
	{}
}
\]
and

\[
\xymatrix@C=2em{
	K \ar[r] &
	Z \ar[r]^p &
	Y\ar@{-->}[r] &
	{}
}
\]
are \(\xi_1\)-triangles with
	\(
	L\in\mathcal{C}_m
	\)
	and
	\(
	K\in\mathcal{F}_1\cap\mathcal{W}_2.
	\)
	
	Since \(p\) is an $m$-trivial fibration, by (4) it is also a $2$-trivial fibration. Hence \(p\) is a $2$-weak equivalence. Since \(f\) is a $2$-weak equivalence, the two-out-of-three property implies that \(i\) is also a $2$-weak equivalence.
	
	Now the \(\xi_2\)-triangle
	\[
	\xymatrix@C=2em{
		X  \ar[r]^i &
		Z \ar[r] &
		L\ar@{-->}[r] &
		{}
	}
	\]
	shows that \(L\in\mathcal{W}_2\). Therefore,
	\(
	L\in\mathcal{C}_m\cap\mathcal{W}_2,
	\)
	so \(i\) is an $m$-trivial cofibration. Consequently, \(f\) is an $m$-weak equivalence.

	Finally, we show that every $1$-weak equivalence is an $m$-weak equivalence. Let $f$ be a $1$-weak equivalence. In the model structure $\mathcal{M}_1$, we can factor $f = p'i'$ where $i'$ is a $1$-trivial cofibration and $p'$ is a $1$-trivial fibration. By definition, $i'$ is a $\xi_1$-inflation with $\mathrm{Cone}(i') \in \mathcal{C}_1 \cap \mathcal{W}_1$, and $p'$ is a $\xi_1$-deflation with $\mathrm{CoCone}(p') \in \mathcal{F}_1 \cap \mathcal{W}_1$. Since $\xi_1 \subseteq \xi_2$ and $\mathcal{W}_1 \subseteq \mathcal{W}_2$, the morphism $i'$ is a $\xi_2$-inflation with its cone in $\mathcal{W}_2$, meaning $i'$ is a $2$-weak equivalence. Similarly, $p'$ is a $\xi_2$-deflation with its cocone in $\mathcal{W}_2$, meaning $p'$ is a $2$-weak equivalence. By the two-out-of-three property for $\mathcal{M}_2$, $f = p'i'$ is a $2$-weak equivalence, and therefore also an $m$-weak equivalence.

\end{proof}

\begin{prop}\label{prop:weak-equivalences-between-mixed-cofibrants}
	An $m$-weak equivalence between two $m$-cofibrant objects is necessarily a $1$-weak equivalence.
\end{prop}

\begin{proof}
	By Proposition~\ref{prop:morphism-comparison}(5), a morphism
	\(
	f:X\to Y
	\)
	is a $2$-weak equivalence if and only if it is an $m$-weak equivalence. Consider a factorization of \(f\) in the mixed model structure \(\mathcal{M}_m\):
	\[
	\begin{tikzcd}[column sep=small]
		X \arrow[rr,"f"] \arrow[dr,"i"']
		&& Y \\
		& Z \arrow[ur,"p"'] &
	\end{tikzcd}
	\]
	where
	
	\[
	\xymatrix@C=2em{
		X  \ar[r]^i &
		Z \ar[r] &
		L\ar@{-->}[r] &
		{}
	}
	\]
	and
	
	\[
	\xymatrix@C=2em{
		K \ar[r] &
		Z \ar[r]^p &
		Y\ar@{-->}[r] &
		{}
	}
	\]
	are \(\xi_1\)-triangles with
	\(
	L\in\mathcal{W}_2\cap\mathcal{C}_m
	=
	\mathcal{W}_1\cap\mathcal{C}_1
	\)
	and
	\(
	K\in\mathcal{F}_1\cap\mathcal{W}_2.
	\)
	
	Since \(Y\in\mathcal{C}_m\), the $\xi_1$-triangle
	\[
\xymatrix@C=2em{
	K \ar[r] &
	Z \ar[r]^p &
	Y\ar@{-->}[r] &
	{}
}
\]
	splits. Hence \(K\) is a direct summand of \(Z\). Since both \(X\) and \(L\) belong to \(\mathcal{C}_m\), and \(\mathcal{C}_m\) is closed under extensions, we obtain
	\(
	Z\in\mathcal{C}_m.
	\)
	Therefore,
	\(
	K\in\mathcal{C}_m.
	\)
	
	Consequently,
	\[
	K\in
	\mathcal{W}_2\cap\mathcal{F}_1\cap\mathcal{C}_m
	=
	\mathcal{W}_1\cap\mathcal{F}_1\cap\mathcal{C}_1.
	\]
	Thus \(p\) is a $1$-trivial fibration, and hence \(f\) is a $1$-weak equivalence.
\end{proof}

With the behavior of morphisms between $m$-cofibrant objects fully understood, we are now in a position to explicitly characterize the objects in $\mathcal{C}_m$.

\begin{prop}[Cofibrant Objects in Mixed Model Structures]\label{prop:cofibrant-characterization}
	Let $(\mathscr{C},\mathbb{E},\mathfrak{s})$ be a weakly idempotent complete extriangulated category, and let $\mathcal{M}_m=(\mathcal{C}_m,\mathcal{W}_2,\mathcal{F}_1)$ be the mixed model structure relative to $\xi_1$ constructed above.
	An object $C$ belongs to $\mathcal{C}_m$ if and only if
	\(C\in\mathcal{C}_1\)
	and there exists a $1$-weak equivalence
	\(C_2 \to C\)
	for some \(C_2\in\mathcal{C}_2\).
	
	Moreover, for every \(C\in\mathcal{C}_m\), there exists a $\xi_2$-deflation
	\(p:C_2\to C\)
	which is a $1$-weak equivalence.
\end{prop}

\begin{proof}
	Assume first that \(A\in\mathcal{C}_m\). By Lemma~\ref{lem:mixed-approximation}, there exists a commutative diagram
	\[
	\xymatrix{
		F_2 \ar[r] \ar[d]
		& C_2 \ar[r]^{p} \ar[d]^{j}
		& A \ar@{=}[d]
		&
		\\
		F_m \ar[r] \ar[d]
		& C_m \ar[r]^{q} \ar[d]
		& A 
		& 
		\\
		C_1 \ar@{=}[r] 
		& C_1 
		&
		&
		\\
		{}
		& 
		&
		&
	}
	\]
	
	Since \(A\in\mathcal{C}_m\), the second row splits. Hence
	\[
	F_m\in
	\mathcal{W}_2\cap\mathcal{F}_1\cap\mathcal{C}_m
	=
	\mathcal{W}_1\cap\mathcal{F}_1\cap\mathcal{C}_1.
	\]
	Therefore \(q\) is a $1$-trivial fibration. From the second column, we know $j$ is a $1$-trivial cofibration, which means $j$ is also a $1$-weak equivalence, so $p: C_2 \to A$ is a $1$-weak equivalence. Furthermore, the first row shows that $p$ is a $\xi_2$-deflation with $C_2 \in \mathcal{C}_2$. This proves both the forward direction and the "Moreover" statement.
	
Conversely, suppose that \(C\in\mathcal{C}_1\) and there exists a $1$-weak equivalence
\(
\alpha:C_2\to C
\)
with
\(
C_2\in\mathcal{C}_2.
\)

Factor \(C\) in the mixed model structure to obtain the following diagram:

\[
\xymatrix{
	& F \ar[d]
	&
	&
	\\
	& C_m \ar[d]^{g}
	&
	&
	\\
	C_2 \ar[r]^{\alpha} \ar@{-->}[ur]^{f}
	& C
	&
	&
	\\
	& 
	&
	&
}
\]
where
\(
F\in\mathcal{W}_2\cap\mathcal{F}_1
\)
and
\(
C_m\in\mathcal{C}_m.
\)

Since \(g\) is an $m$-trivial fibration, Proposition~\ref{prop:morphism-comparison}(4) implies that \(g\) is also a $2$-trivial fibration, and hence a $2$-weak equivalence.

Now,
\(
\mathbb{E}_{1}(C_2,F)=0
\)
because
\(
C_2\in\mathcal{C}_2\subseteq\mathcal{C}_m
\)
and
\(
F\in\mathcal{W}_2\cap\mathcal{F}_1.
\)
Hence there exists a morphism
\(
f:C_2\to C_m
\)
such that
\(
gf=\alpha.
\)

Since both \(g\) and \(\alpha\) are $2$-weak equivalences, so is \(f\). By Proposition~\ref{prop:weak-equivalences-between-mixed-cofibrants}, \(f\) is in fact a $1$-weak equivalence. Therefore \(g\) is a $1$-trivial fibration, which implies
\[
F\in\mathcal{W}_1\cap\mathcal{F}_1.
\]

Thus the $\xi_1$-triangle

	\[
\xymatrix@C=2em{
	F \ar[r]^i &
	C_m \ar[r] &
	C\ar@{-->}[r] &
	{}
}
\]
splits. Consequently, \(C\) is a direct summand of \(C_m\). Since \(\mathcal{C}_m\) is closed under direct summands, we conclude that
\(
C\in\mathcal{C}_m.
\)
This completes the proof.

\end{proof}

\section{Applications}\label{sec:applications}

Theorem~\ref{prop:key_prop} is stated for an arbitrary weakly idempotent
complete extriangulated category equipped with two proper classes
$\xi_1\subseteq\xi_2$. In this section we specialize it to exact
categories, where proper classes are the same as exact substructures
(Subsection~\ref{subsec:exact}), and to triangulated categories, where
they are proper classes of triangles (Subsection~\ref{subsec:triangulated}).
In both settings the only task is to check that the hypotheses of
Theorem~\ref{prop:key_prop} are satisfied. We close with an explicit
example on the homotopy category of a quasi-Frobenius ring, in which
$\xi_1\subsetneq\xi_2$ and the mixed model structure differs from both
$\mathcal{M}_1$ and $\mathcal{M}_2$ (Subsection~\ref{subsec:example}).

We first record a general observation.

\begin{rmk}\label{rem:same-homotopy-category}
In the situation of Theorem~\ref{prop:key_prop}(1), the model structures
$\mathcal{M}_m$ and $\mathcal{M}_2$ are defined on the same category and,
by Proposition~\ref{prop:morphism-comparison}(5), have the same weak
equivalences; the same holds, dually, in the situation of
Theorem~\ref{prop:key_prop}(2). Hence in both cases
\[
\Ho(\mathcal{M}_m)=\Ho(\mathcal{M}_2).
\]
Thus mixing does not change the homotopy category of $\mathcal{M}_2$.
What changes is the proper class with respect to which the model structure
is admissible (namely $\xi_1$ instead of $\xi_2$), together with the
classes of cofibrant and fibrant objects.
\end{rmk}

\subsection{Exact categories}\label{subsec:exact}

Let $(\mathcal{A},\mathcal{E})$ be an exact category, viewed as an
extriangulated category with
$\mathbb{E}=\Ext^1_{\mathcal{E}}(-,-)$ \cite{NP}.
Its admissible model structures are precisely the exact model structures
of \cite{gillespie2011model}. An \emph{exact substructure} is an exact
structure $\mathcal{E}'\subseteq\mathcal{E}$ on $\mathcal{A}$.

\begin{lem}[{\cite[Remark 3.3(2)]{hu2020proper}}]
	\label{lem:exact-substructures}
	Let $(\mathcal{A},\mathcal{E})$ be an exact category, regarded as an
	extriangulated category as above, and let $\xi\subseteq\mathcal{E}$
	be a class of conflations closed under isomorphisms.
	Then $\xi$ is a proper class of $\mathbb{E}$-triangles if and only if
	$(\mathcal{A},\mathbb{E}_\xi,\mathfrak{s}_\xi)$ is an exact category.
\end{lem}

The following corollary recovers Gillespie's mixing theorem for exact
model structures \cite[Chapter~11]{Gillespie_2025}; we omit the proof.

\begin{cor}[Cole's theorem for exact model structures]\label{cor:exact-mixing}
	Let $(\mathcal{A},\mathcal{E})$ be a weakly idempotent complete exact
	category with two exact substructures $\mathcal{E}_1\subseteq\mathcal{E}_2$
	of $\mathcal{E}$. Let
	$\mathcal{M}_1=(\mathcal{C}_1,\mathcal{W}_1,\mathcal{F}_1)$ and
	$\mathcal{M}_2=(\mathcal{C}_2,\mathcal{W}_2,\mathcal{F}_2)$ be exact model
	structures on $(\mathcal{A},\mathcal{E}_1)$ and
	$(\mathcal{A},\mathcal{E}_2)$, respectively, such that
	$\mathcal{W}_1\subseteq\mathcal{W}_2$.
	\begin{enumerate}
		\item If $\mathcal{F}_1\subseteq\mathcal{F}_2$, then there is an exact
		model structure $\mathcal{M}_m=(\mathcal{C}_m,\mathcal{W}_2,\mathcal{F}_1)$
		on $(\mathcal{A},\mathcal{E}_1)$, where
		$\mathcal{C}_m={}^{\perp_1}(\mathcal{W}_2\cap\mathcal{F}_1)$. An object
		$C$ lies in $\mathcal{C}_m$ if and only if $C\in\mathcal{C}_1$ and there
		is a $1$-weak equivalence $C_2\to C$ with $C_2\in\mathcal{C}_2$;
		moreover, this $1$-weak equivalence can be chosen to be an
		$\mathcal{E}_2$-deflation.
		\item If $\mathcal{C}_1\subseteq\mathcal{C}_2$, then there is an exact
		model structure $\mathcal{M}_m=(\mathcal{C}_1,\mathcal{W}_2,\mathcal{F}_m)$
		on $(\mathcal{A},\mathcal{E}_1)$, where
		$\mathcal{F}_m=(\mathcal{C}_1\cap\mathcal{W}_2)^{\perp_1}$.
	\end{enumerate}
	In both cases the weak equivalences of $\mathcal{M}_m$ are exactly those
	of $\mathcal{M}_2$.
\end{cor}

The following standard example illustrates Cole's mixing construction
\cite{cole2006mixing} for chain complexes; see also
\cite[Chapter~11]{Gillespie_2025}.

\begin{example}[Hurewicz and projective model structures on chain complexes]
	\label{ex:chain-complexes}\rm
	Let $R$ be a ring, $\mathcal{A}=\mathrm{Ch}(R)$, $\mathcal{E}_1=\mathcal{E}_{\rm dw}$
	the degreewise split short exact sequences and $\mathcal{E}_2=\mathcal{E}$ all
	short exact sequences. Let $\mathcal{M}_1=(\mathrm{Ch}(R),\mathsf{Ctr},\mathrm{Ch}(R))$
	be the Hurewicz model structure on $(\mathrm{Ch}(R),\mathcal{E}_{\rm dw})$ and
	$\mathcal{M}_2=(\mathsf{dgP},\mathsf{Ac},\mathrm{Ch}(R))$ the projective model
	structure on $(\mathrm{Ch}(R),\mathcal{E})$, where $\mathsf{Ctr}$, $\mathsf{Ac}$
	and $\mathsf{dgP}$ denote the contractible, acyclic and dg-projective complexes
	\cite{gillespie2011model,Gillespie_2025}. Corollary~\ref{cor:exact-mixing}(1)
	applies, and since
	$\Ext^1_{\mathcal{E}_{\rm dw}}(X,Y)\cong\Hom_{\mathrm{K}(R)}(X,Y[1])$, the class
	$\mathcal{C}_m={}^{\perp_1}\mathsf{Ac}$ consists of the K-projective complexes
	\cite{Spaltenstein1988}. Thus we recover the well-known fact that a complex is
	K-projective if and only if it admits a surjective chain homotopy equivalence
	from a dg-projective complex.
\end{example}

\subsection{Triangulated categories}\label{subsec:triangulated}

Let $(\mathcal{T},[1],\Delta)$ be a triangulated category, where $\Delta$
denotes the class of distinguished triangles. Setting
$\mathbb{E}(C,A)=\Hom_{\mathcal{T}}(C,A[1])$ and
$\mathfrak{s}(\delta)=[A\to B\to C]$ for any distinguished triangle
$A\to B\to C\overset{\delta}{\to}A[1]$, one obtains an extriangulated
category whose $\mathbb{E}$-triangles are the distinguished triangles
\cite{NP}. Since every morphism of $\mathcal{T}$ can be completed to a
distinguished triangle, every morphism is an inflation. Hence condition~(3)
of Lemma~\ref{lem:wic} holds trivially, and $\mathcal{T}$ is weakly
idempotent complete.

By a \emph{proper class of triangles} we mean a class $\xi\subseteq\Delta$
that is proper in the sense of Definition~\ref{def:proper class} and closed
under all shifts, i.e.\ $\xi[n]=\xi$ for every $n\in\mathbb{Z}$
\cite[Section~2.2]{Beligiannis2000}; cf.\
\cite[Remark~3.3(3)]{hu2020proper}.
Examples include $\Delta$, the class $\Delta_0$ of split triangles, and
the class of pure triangles in a compactly generated triangulated category
\cite{Beligiannis2000,KrauseSmashing}.
For a proper class $\xi$, the triple
$(\mathcal{T},\mathbb{E}_\xi,\mathfrak{s}_\xi)$ is a weakly idempotent
complete extriangulated category by Theorem~\ref{thm:proper-extri} and
Proposition~\ref{prop:wic-relative}; for $\xi\neq\Delta$ it need not be
exact or triangulated. We refer to admissible model structures relative to
$\xi$ (Definition~\ref{def:admissible}) as \emph{admissible model
structures on $\mathcal{T}$ relative to $\xi$}. For $\xi=\Delta$ these are
the admissible model structures on the triangulated category $\mathcal{T}$;
cf.\ \cite{yang2015model}.

\begin{cor}[Cole's theorem relative to proper classes of triangles]
	\label{cor:triangulated-mixing}
	Let $\mathcal{T}$ be a triangulated category, and let
	$\xi_1\subseteq\xi_2$ be proper classes of triangles in $\mathcal{T}$. Let
	$\mathcal{M}_1=(\mathcal{C}_1,\mathcal{W}_1,\mathcal{F}_1)$ and
	$\mathcal{M}_2=(\mathcal{C}_2,\mathcal{W}_2,\mathcal{F}_2)$ be admissible
	model structures on $\mathcal{T}$ relative to $\xi_1$ and $\xi_2$,
	respectively, such that $\mathcal{W}_1\subseteq\mathcal{W}_2$.
	\begin{enumerate}
		\item If $\mathcal{F}_1\subseteq\mathcal{F}_2$, then there is an
		admissible model structure
		$\mathcal{M}_m=(\mathcal{C}_m,\mathcal{W}_2,\mathcal{F}_1)$ on
		$\mathcal{T}$ relative to $\xi_1$, where
		$\mathcal{C}_m={}^{\perp_1}(\mathcal{W}_2\cap\mathcal{F}_1)$. An object
		$C$ lies in $\mathcal{C}_m$ if and only if $C\in\mathcal{C}_1$ and there
		is a $1$-weak equivalence $C_2\to C$ with $C_2\in\mathcal{C}_2$;
		moreover, this $1$-weak equivalence can be chosen to be a
		$\xi_2$-deflation.
		\item If $\mathcal{C}_1\subseteq\mathcal{C}_2$, then there is an
		admissible model structure
		$\mathcal{M}_m=(\mathcal{C}_1,\mathcal{W}_2,\mathcal{F}_m)$ on
		$\mathcal{T}$ relative to $\xi_1$, where
		$\mathcal{F}_m=(\mathcal{C}_1\cap\mathcal{W}_2)^{\perp_1}$.
	\end{enumerate}
	In both cases the weak equivalences of $\mathcal{M}_m$ are exactly those
	of $\mathcal{M}_2$; in particular, $\Ho(\mathcal{M}_m)=\Ho(\mathcal{M}_2)$.
\end{cor}

\begin{proof}
	As observed above, $\mathcal{T}$ is weakly idempotent complete. Hence
	Theorem~\ref{prop:key_prop},
	Proposition~\ref{prop:cofibrant-characterization} and
	Proposition~\ref{prop:morphism-comparison}(5), together with their duals,
	apply.
\end{proof}

\begin{rmk}\label{rem:role-of-proper-classes}
	(1) If $\xi_1=\xi_2=\xi$, then Corollary~\ref{cor:triangulated-mixing} is
	Cole's mixing theorem inside the single extriangulated category
	$(\mathcal{T},\mathbb{E}_\xi,\mathfrak{s}_\xi)$; for $\xi=\Delta$ it is a
	mixing theorem for admissible model structures on $\mathcal{T}$ itself.
	The phenomenon specific to Theorem~\ref{prop:key_prop}, namely the
	interaction of the two orthogonality relations $\perp_1$ and $\perp_2$
	handled in Lemmas~\ref{lem:mixed-approximation}
	and~\ref{lem:extension-comparison}, appears only when
	$\xi_1\subsetneq\xi_2$.
	
	(2) The smallest proper class gives nothing new. Indeed, if
	$\xi_1=\Delta_0$, then $\mathbb{E}_{\xi_1}=0$, so the only cotorsion pair
	relative to $\xi_1$ is $(\mathcal{T},\mathcal{T})$. Hence the only Hovey
	triple relative to $\xi_1$ is $(\mathcal{T},\mathcal{T},\mathcal{T})$, and
	by Theorem~\ref{thm:relative-hovey} the only admissible model structure
	relative to $\xi_1$ is $(\mathcal{T},\mathcal{T},\mathcal{T})$, in which
	every morphism is a weak equivalence. In this case necessarily
	$\mathcal{M}_m=\mathcal{M}_1=(\mathcal{T},\mathcal{T},\mathcal{T})$. The
	same argument applies in any extriangulated category.
	
	Thus the interesting instances of Corollary~\ref{cor:triangulated-mixing}
	are those with $\Delta_0\subsetneq\xi_1\subsetneq\xi_2$. The example in
	Subsection~\ref{subsec:example} is of this kind.
\end{rmk}

\subsection{A genuinely extriangulated example}\label{subsec:example}

We construct a non-degenerate mixed model structure on an extriangulated
category which is neither exact nor triangulated. We first recall the
ghost proper class used below.

\begin{definition}\label{def:ghost-class}\rm
	Let $R$ be a ring. A morphism $f$ in $\mathrm{K}(R)$ is called a
	\emph{ghost} if $H^n(f)=0$ for every $n\in\mathbb{Z}$. Denote by
	$\xi_{\mathrm{ghost}}$ the class of distinguished triangles
	$X\to Y\to Z\xrightarrow{\delta}X[1]$ whose connecting morphism
	$\delta$ is a ghost.
\end{definition}

By \cite[Example~2.3(3)]{Beligiannis2000}, $\xi_{\mathrm{ghost}}$ is a
proper class; equivalently, its triangles induce short exact sequences
in every cohomological degree.

\begin{example}[Mixing over a quasi-Frobenius ring]\label{ex:QF}\rm
	Let $R$ be a non-semisimple quasi-Frobenius ring, e.g.\
	$R=k[x]/(x^2)$ for a field $k$, and let
	$\xi=\xi_{\mathrm{ghost}}$ in $\mathrm{K}(R)$.
	Write $\mathsf{Ac}$ and $\mathsf{KInj}$ for the acyclic and
	K-injective complexes, respectively, and set
	\[
		\mathsf{P}_H=
		\left\{Q\in\mathrm{K}(R)\ \middle|\
		Q\cong_{\mathrm{K}(R)}\bigoplus_{n\in\mathbb{Z}}P^n[-n],
		\ P^n\in\operatorname{Proj}R\right\}.
	\]
	
	The standard Cartan--Eilenberg descriptions yield, in
	$\mathrm{K}(R)$, enough $\xi$-projectives and $\xi$-injectives,
	both given by $\mathsf{P}_H$; cf.\
	\cite{EstradaFuHerzogOdabasi2026}.
	Thus Theorem~\ref{thm:relative-hovey} and
	\cite{Spaltenstein1988} give admissible model structures
	\[
	\begin{gathered}
		\mathcal{M}_1=
		(\mathrm{K}(R),\mathsf{P}_H,\mathrm{K}(R))
		\quad\text{relative to }\xi,\ \ 
		\mathcal{M}_2=
		(\mathsf{Ac},\mathsf{KInj},\mathrm{K}(R))
		\quad\text{relative to }\Delta.
	\end{gathered}
	\]
	As $\mathsf{P}_H\subseteq\mathsf{KInj}$,
	Corollary~\ref{cor:triangulated-mixing}(1), with
	$\xi_1=\xi$ and $\xi_2=\Delta$, gives
	\[
	\begin{gathered}
		\mathcal{M}_m=
		(\mathcal{C}_m,\mathsf{KInj},\mathrm{K}(R))
		\quad\text{relative to }\xi,\\
		\mathcal{C}_m=
		\{X\in\mathrm{K}(R)\mid
		H^n(X)\text{ is projective for all }n\}.
	\end{gathered}
	\]
	Indeed, $X\in\mathcal{C}_m$ if and only if every ghost
	$X\to I$ with $I\in\mathsf{KInj}$ vanishes.
	Since localization preserves morphisms into K-injective
	complexes, this is equivalent to the image of $X$ in
	$\mathrm{D}(R)$ being ghost-projective, or equivalently,
	to each $H^n(X)$ being projective \cite{Yang2026Induced}.
	For such $X$, lifting $H^n(X)$ to cycles gives a
	quasi-isomorphism $Q=\bigoplus_nH^n(X)[-n]\to X$.
	Its cone triangle splits since the cone is acyclic and
	$Q\in\mathsf{KInj}$; hence
	$X\cong Q\oplus A'$ with $Q\in\mathsf{P}_H$ and
	$A'\in\mathsf{Ac}$, yielding
	$\mathcal{C}_m\cap\mathsf{KInj}=\mathsf{P}_H$.
	
	By Proposition~\ref{prop:morphism-comparison}, the fibrations
	of $\mathcal{M}_m$ are the cohomology-surjective morphisms
	and its weak equivalences are the morphisms with K-injective
	mapping cone. Thus
	\[
	\Ho(\mathcal{M}_m)=\Ho(\mathcal{M}_2)
	\simeq\mathrm{K}_{\rm ac}(R)\neq0.
	\]
	Since $R[0]\in\mathcal{C}_m\setminus\mathsf{Ac}$ and
	$M[0]\notin\mathcal{C}_m$ for a non-projective module $M$,
	we have
	$\mathcal{C}_2\subsetneq\mathcal{C}_m\subsetneq\mathcal{C}_1$.
	
	Fix a nonzero object $A\in\mathsf{Ac}$ of $\mathrm{K}(R)$,
	e.g.\ a complete projective resolution of a non-projective module.
	The extriangulated category
	$(\mathrm{K}(R),\mathbb{E}_{\xi},\mathfrak{s}_{\xi})$
	carrying $\mathcal{M}_m$ is \emph{neither exact nor triangulated}:
	$A\to0$ is a $\xi$-inflation but not a monomorphism, whereas
	$R[0]\to0$ is not a $\xi$-inflation, although every morphism
	is an inflation in the extriangulated structure associated
	to a triangulated category.
	These also show that $\Delta_0\subsetneq\xi\subsetneq\Delta$.
	
\end{example}

\vspace{1cm}

\begin{flushleft}
	\textbf{Junpeng Ren}\\
	School of Mathematics and Statistics, Northeast Normal University\\
	Changchun 130024, China\\
	E-mail: renjp@nenu.edu.cn
\end{flushleft}

\begin{flushleft}
	\textbf{Xianhui Fu}\\
	School of Mathematics and Statistics, Northeast Normal University\\
	Changchun 130024, China\\
	E-mail: fuxianhui@gmail.com
\end{flushleft}

\end{document}